\documentclass[reqno]{amsart}

\usepackage{amscd,amsmath,amssymb,amsfonts}
\usepackage{mathrsfs}
\usepackage[dvips]{color} 
\usepackage[all]{xy}

\numberwithin{equation}{section}

\theoremstyle{plain}
    \newtheorem{thm}{Theorem}[section]
    \newtheorem{lem}[thm]{Lemma}
    
    \newtheorem{prop}[thm]{Proposition}

\theoremstyle{definition}    
    
    \newtheorem{exmp}[thm]{Example}
    
\begin{document}
\title{Explicit logarithmic formulas of special values of hypergeometric functions ${}_3F_2$}
\author{Masanori Asakura and Toshifumi Yabu}
\address{Department of Mathematics, Hokkaido University, Sapporo, 060-0810 Japan}
\email{asakura@math.sci.hokudai.ac.jp}
\address{Department of Mathematics, Hokkaido University, Sapporo, 060-0810 Japan}
\email{}

\maketitle

\begin{abstract}
In the paper \cite{aot-1},
we proved that the value of ${}_3F_2\left({a,b,q\atop a+b,q};1\right)$ of the
generalized hypergeometric function is a $\overline{\mathbb Q}$-linear combination of log of
algebraic numbers if rational numbers $a,b,q$ satisfy a certain condition.
In this paper, we present a method to obtain an explicit description of it.
\end{abstract}

\def\ch{{\mathrm{ch}}}
\def\Coker{\mathrm{Coker}}
\def\crys{\mathrm{crys}}
\def\dlog{d{\mathrm{log}}}
\def\dR{{\mathrm{d\hspace{-0.2pt}R}}}            
\def\et{{\mathrm{\acute{e}t}}}  
\def\Frac{{\mathrm{Frac}}}
\def\id{{\mathrm{id}}}              
\def\Image{{\mathrm{Im}}}        
\def\Hom{{\mathrm{Hom}}}  
\def\Ext{{\mathrm{Ext}}}
\def\MHS{{\mathrm{MHS}}}  
  
\def\ker{{\mathrm{Ker}}}          
\def\Pic{{\mathrm{Pic}}}
\def\CH{{\mathrm{CH}}}
\def\NS{{\mathrm{NS}}}
\def\NF{{\mathrm{NF}}}
\def\End{{\mathrm{End}}}
\def\pr{{\mathrm{pr}}}
\def\Proj{{\mathrm{Proj}}}
\def\ord{{\mathrm{ord}}}
\def\qis{{\mathrm{qis}}}
\def\reg{{\mathrm{reg}}}          %
\def\res{{\mathrm{res}}}          %
\def\Res{\mathrm{Res}}
\def\Spec{{\mathrm{Spec}}}     
\def\can{{\mathrm{can}}}
\def\cont{{\mathrm{cont}}}
\def\zar{{\mathrm{zar}}}
\def\Tr{{\mathrm{Tr}}}
\def\tr{{\mathrm{tr}}}
\def\bA{{\mathbb A}}
\def\bC{{\mathbb C}}
\def\C{{\mathbb C}}
\def\G{{\mathbb G}}
\def\bE{{\mathbb E}}
\def\bF{{\mathbb F}}
\def\F{{\mathbb F}}
\def\bG{{\mathbb G}}
\def\bH{{\mathbb H}}
\def\bJ{{\mathbb J}}
\def\bL{{\mathbb L}}
\def\cL{{\mathscr L}}
\def\bN{{\mathbb N}}
\def\bP{{\mathbb P}}
\def\P{{\mathbb P}}
\def\bQ{{\mathbb Q}}
\def\Q{{\mathbb Q}}
\def\bR{{\mathbb R}}
\def\R{{\mathbb R}}
\def\bZ{{\mathbb Z}}
\def\Z{{\mathbb Z}}
\def\cA{{\mathscr A}}
\def\cH{{\mathscr H}}
\def\cM{{\mathscr M}}
\def\cD{{\mathscr D}}
\def\cE{{\mathscr E}}
\def\cO{{\mathscr O}}
\def\O{{\mathscr O}}
\def\cR{{\mathscr R}}
\def\cS{{\mathscr S}}
\def\cX{{\mathscr X}}
%
\def\ep{\epsilon}
\def\vG{\varGamma}
\def\vg{\varGamma}
%
%
%
%
\def\lra{\longrightarrow}
\def\lla{\longleftarrow}
\def\Lra{\Longrightarrow}
\def\hra{\hookrightarrow}
\def\lmt{\longmapsto}
\def\ot{\otimes}
\def\op{\oplus}
\def\wt#1{\widetilde{#1}}
\def\wh#1{\widehat{#1}}
\def\spt{\sptilde}
\def\ol#1{\overline{#1}}
\def\ul#1{\underline{#1}}
\def\us#1#2{\underset{#1}{#2}}
\def\os#1#2{\overset{#1}{#2}}

\def\Aut{\mathrm{Aut}}
\def\rank{\mathrm{rank}}
\def\fib{\mathrm{fib}}
\def\gen{\mathrm{gen}}
\def\Del{\mathrm{Del}}


\section{Introduction}\label{intro-sect}
The (generalized) hypergeometric function is defined to be the complex analytic function
\[
{}_{p+1}F_p\left({a_1,\ldots,a_{p+1}\atop b_1,\ldots,b_p};x\right)=\sum_{n=0}^\infty
\frac{(a_1)_n\cdots(a_{p+1})_n}{(b_1)_n\cdots(b_p)_n}\frac{x^n}{n!}
\]
where $(\alpha)_n=\alpha\cdot(\alpha+1)\cdots(\alpha+n-1)$ denotes 
the Pochhammer symbol. We refer to the books \cite{Bailey}, \cite{bateman} or \cite{slater}
for the general theory of hypergeometric functions.
The most classical case is the case $p=1$, which is often called the
Gauss hypergeometric function.
A number of formulas on the hypergeometric functions are known.
For example,  Gauss proved that the value of ${}_2F_1$ at $x=1$ is given by
the product of Gamma values (e.g. \cite{Bailey} 1.3)
\[
{}_2F_1\left({a,b\atop c};1\right)=\frac{\Gamma(c)\Gamma(c-a-b)}{\Gamma(c-a)\Gamma(c-b)}
,\quad\mathrm{Re}(c-a-b)>0.\]
One also finds a number of generalizations for ${}_{p+1}F_p$ in \cite{NIST} 16.4.
In the paper \cite{aot-1},
we provided a new formula on the value of ${}_3F_2$ at $x=1$.
\begin{thm}[Log formula, \cite{aot-1}] \label{log-formula}
For $x\in\R$, let $\{x\}:=x-\lfloor x\rfloor$ denote the decimal part.
Let $a,b,q\in\Q$ be non-integers such that
none of $q-a,q-b,q-a-b$ is an integer.
Assume that
\begin{equation}\label{main-cond}
\{sq\}+\{s(-q+a)\}+\{s(-q+b)\}+\{s(q-a-b)\}=2
\end{equation}
holds for all $s\in \Z$ prime to the denominators of $a,b,q$.
Then
\begin{equation}\label{main-result}
B(a,b){}_3F_2\left(
\begin{matrix}
a,b,q\\
a+b,q+1\end{matrix};1
\right)
\in\ol{\Q}+\ol{\Q}\log\ol{\Q}^\times.
\end{equation}
Here $B(a,b)=\Gamma(a)\Gamma(b)/\Gamma(a+b)$ is the beta function, and
the right hand side denotes the $\ol\Q$-linear subspace of $\C$ generated by
$1$, $2\pi i$ and 
$\log \alpha$'s, $\alpha\in\ol{\Q}^\times$.
\end{thm}
There remains a question to obtain an explicit description of \eqref{main-result}
(which we call {\it explicit log formula}),
and it has not been completed except some cases.

\medskip

The purpose of this paper is to present a general method for the explicit log formula.
The key ingredient is the {\it Beilinson regulator} and the {\it hypergeometric fibration} introduced by
Otsubo and the first author in \cite{a-o-2}.
For example, we discuss
the fibration $f_l:X_l\to \P^1$
whose general fiber $f_l^{-1}(t)$ is the curve
\[
y^N=x^A(1-x)^B(1-t^lx)^{N-B}
\]
where $N,A,B,l$ are positive integers such that $0<A,B<N$.
Though most part of our method follows the argument in \cite{aot-1},
we need to employ a new technique developed in \cite{a-real} (see also \cite{a-o-1} Appendix), 
namely constructing a certain ``rational differential 2-form'', which
we denote by $\omega_\Del$ (see \S \ref{del-sect} for definition).
There still remains a difficulty to work out the explicit log formula. 
We need to know generators of the Neron-Severi group of $X_l$ explicitly
(see \S \ref{ELF-sect} for detail). This is done in some cases, while it seems very hard in many other cases.

\medskip

This paper is organized as follows.
In \S \ref{gen-sect} we give a general method for explicit log formulas.
The main theorem is Theorem \ref{ex-thm-1}.
In \S \ref{rec-sect}, we demonstrate how to apply Theorem \ref{ex-thm-1} 
and how to obtain explicit log formulas
in the case $(a,b,q)=(\frac{1}{6},\frac{5}{6},\frac{1}{2})$.
We also give explicit log formulas (without proof) in the cases
$(a,b,q)=(\frac{1}{6},\frac{5}{6},\frac{i}{3})$,$(\frac{1}{6},\frac{5}{6},\frac{j}{4})$ and
$(\frac{1}{6},\frac{5}{6},\frac{k}{5})$ with $i\in\{1,2\}$, $j\in\{1,2,3\}$ and $k\in\{1,\ldots,4\}$.

\medskip

Finally we note that Terasoma recently developed 
a different method from ours, and
obtained explicit log formulas in many cases \cite{terasoma}.
For example, the cases $(a,b)=(\frac{1}{6},\frac{5}{6})$ and $q=\frac{1}{2},\frac{i}{3},\frac{j}{4}$
are covered by his.
On the other hand, the case $(a,b,q)=(\frac{1}{6},\frac{5}{6},\frac{k}{5})$ is not covered,
both methods have own advantages.

There remains the question on explicit description 
of the {\it functional log formula} proved in \cite{a-o-log}.
We expect that our method of hypergeometric fibration shall also work.

\noindent{\bf Acknowledgement.}
The authors are grateful to the referee for reading the manuscript carefully 
and pointing out lots of errors.

\section{Sketch of Proof of Log Formula \cite{aot-1}}\label{sketch-sect}
In the paper \cite{aot-1},
we gave two proofs of the log formula (Theorem \ref{log-formula}).
One uses the hypergeometric fibrations and the other does the Fermat surfaces.
The crucial point is to relate the special values of ${}_3F_2$ to
the Beilinson regulator of certain elements of motivic cohomology $H^3_\cM(X,\Z(2))$.
In this section we review the former proof using hypergeometric fibrations.
The explicit log formula shall be obtained by improving it.

\medskip

Throughout this paper, we fix an embedding $\ol\Q\hra\C$.

\subsection{Hypergeometric Fibrations}\label{HG-sect}
We recall the hypergeometric fibrations introduced in \cite{a-o-2} \S 3.1. 
Let $R$ be a finite-dimensional semisimple $\Q$-algebra.
Let $e:R\to E$ be a projection onto a number field $E$.
Let $X$ be a smooth projective variety over $k_\dR$, and $f:X\to \P^1$
a surjective map endowed with a multiplication on $R^1f_*\Q|_U$ by $R$
where $U\subset \P^1$ is the maximal Zariski open set such that $f$ is smooth over $U$. 
We say $f$ is a {\it hypergeometric fibration with multiplication by $(R,e)$} (abbreviated HG
fibration) if
the following conditions hold.
We fix an inhomogeneous coordinate $t\in\P^1$.
\begin{enumerate}
\item[\bf (a)]
$f$ is smooth over $\P^1\setminus\{t=0,1,\infty\}$,
\item[\bf (b)]
$\dim_E (R^1f_*\Q)(e)=2$ where we write $V(e):=E\ot_{e,R}V$ the $e$-part,
\item[\bf (c)]
Let $\Pic_f^0\to \P^1\setminus\{0,1,\infty\}$ be the Picard fibration whose general fiber
is the Picard variety $\Pic^0(f^{-1}(t))$,
and let $\Pic_f^0(e)$ be the component associated to the $e$-part $(R^1f_*\Q)(e)$
(this is well-defined up to isogeny).
Then $\Pic_f^0(e)\to\P^1\setminus\{0,1,\infty\}$ has 
totally degenerate semistable reduction at $t=1$.
\end{enumerate}
The last condition {\bf (c)} is equivalent to saying that
the local monodromy $T$ on $(R^1f_*\Q)(e)$ at $t=1$ is unipotent
and the rank of log monodromy $N:=\log(T)$ is maximal, namely 
$\rank(N)=\frac{1}{2}\dim_\Q (R^1f_*\Q)(e)$
($=[E:\Q]$ by the condition {\bf(b)}).
\begin{exmp}\label{HG-exmp}
Let $f:X\to \P^1$ be an elliptic fibration.
Then $f$ is a HG fibration with multiplication by $(\Q,\mathrm{id})$
if and only if $f$ is smooth over $\P^1\setminus\{0,1,\infty\}$
and the reduction at $t=1$ is multiplicative (i.e. of type $I_n$, $n>0$).
\end{exmp}
\begin{exmp}[\cite{a-o-2} \S 3.2]\label{HG-exmp-gauss}
Let $N,A,B$ be integers such that $0<A,B<N$ and $\gcd(A,N)=\gcd(B,N)=1$.
Let $f:X\to \P^1$ be a fibration whose general fiber $X_t=f^{-1}(t)$ is the projective
nonsingular model of an affine curve 
\[
y^N=x^A(1-x)^B(1-tx)^{N-B}.
\]
Then $f$ is smooth over $\P^1\setminus\{t=0,1,\infty\}$.
Let $\mu_N$ be the group of $N$-th roots of unity.
For $\zeta_N\in\mu_N$,
the automorphism given by $(x,y,t)\mapsto(x,\zeta_Ny,t)$ gives rise to the multiplication
by the group ring $R=\Q[\mu_N]$.
Let $e:R\to E$ be a projection onto a number field $E$.
If $E\ne\Q$, then $(R,e)$ satisfies the conditions
{\bf(b), (c)}.
We call $f$ the {\it HG fibration of Gauss type}.
\end{exmp}
\subsection{Motivic cohomology and Deligne-Beilinson cohomology}
The theory of the motivic cohomology groups
\[
H^i_\cM(X,\Z(j))
\] 
of a variety $X$ over a field is developed by Suslin, Voevodsky et al.
We here review $H^3_\cM(X,\Z(2))$, which has
an elementary description in the following way.
Let $X$ be a smooth quasi-projective variety over a field $k$.
We denote by $K_2^M$ the Milnor $K$-theory.
Then the {\it motivic cohomology group} $H^3_\cM(X,\Z(2))$
can be identified with the cohomology at the middle term of of the following complex
\begin{equation}\label{delta}
K_2^M(\ol\Q(X))\os{\delta_2}{\lra} \bigoplus_Dk(D)^\times 
\os{\delta_1}{\lra} \bigoplus_E\Z
\end{equation}
at the middle term, where
$D$ and $E$ run over all integral closed subschemes on $X$ of codimension $1$ and $2$ respectively, and $\delta_i$ are given as follows
\[
\delta_2\{f,g\}=\sum_D(-1)^{v_D(f)v_D(g)}\frac{f^{v_D(g)}}{g^{v_D(f)}}|_D,\quad
\delta_1\left(\sum_D(f,D)\right)=\sum_D\mathrm{div}_D(f).
\]
Here
$(f,D)$ denotes an element $f\in k(D)^\times\subset \op_Dk(D)^\times$
placed in the $D$-component.
Thus any element of $H^3_\cM(X,\Z(2))$ is represented by an element
$\sum_D(f,D)$ satisfying $\sum_D\mathrm{div}_D(f)=0$. 
Note that the Chow group $\CH^2(X)$ is defined to be the cokernel  of $\delta_1$.
For a closed subscheme $Z\subset X$ of codimension $1$, 
the motivic cohomology  $H^3_{\cM,Z}(X,\Z(2))$
supported on $Z$ is canonically isomorphic to the kernel of
\[
\bigoplus_{D\subset Z}k(D)^\times 
\os{\delta_1}{\lra} \bigoplus_{E\subset Z}\Z.
\]
Hence there is an exact sequence
\[
H^3_{\cM,Z}(X,\Z(2))\to H^3_\cM(X,\Z(2))\to
H^3_\cM(X\setminus Z,\Z(2)).
\]
Let $X$ be a projective smooth variety over $\C$, and $Z\subset X$ a closed subscheme.
The {\it Deligne-Beilinson cohomology} group
$H^\bullet_{\cD,Z}(X,\Z(r))$ is defined to be the cohomology
${\mathbb H}^\bullet_Z(X^{an},\Z(r)_\cD)$ of the complex
\[
\Z(r)_\cD:\Z(r)\to \O_X\to\Omega^1_X\to\cdots\to\Omega^{r-1}_X
\]
of sheaves on the analytic site $X^{an}$ (e.g. \cite{ev}). 
Write $H^\bullet_{\cD}(X,\Z(r)):=H^\bullet_{\cD,X}(X,\Z(r))$.
If the base field is $\ol\Q$,
we simply write $H^\bullet_{\cD,Z}(X,\Z(r))=H^\bullet_{\cD,Z\times_{\ol\Q}\C}(X\times_{\ol\Q}\C,\Z(r))$
(note that we fix an embedding $\ol\Q\hra \C$ throughout the paper). 
There is the {\it Beilinson regulator map} (or higher Chern class map)
\begin{equation}\label{regulator-H32}
\xymatrix{
\reg:H^i_{\cM,Z}(X,\Z(r))\ar[r]&
H^i_{\cD,Z}(X,\Z(r)).
}
\end{equation}
We refer to \cite{schneider} for the definition of regulator maps.
We shall discuss the case $(i,r)=(3,2)$ in detail in \S \ref{bei-sect}.
There is the exact sequence
\[
0\to H^2_B(X,\C)/F^2H^2_B(X,\C)+ H^2_B(X,\Z(2))
\to
H^3_\cD(X,\Z(2))
\os{i}{\to} H^3_B(X,\Z(2))_{\mathrm{tor}}
\to 0
\]
where $F^\bullet$ denotes the Hodge filtration.
Write $H^3_\cD(X,\Z(2))':=\ker(i)$. One has
\begin{align}
H^3_\cD(X,\Z(2))'
&\cong H^2_B(X,\C)/F^2H^2_B(X,\C)+ H^2_B(X,\Z(2))\label{DB-isom-1}\\
&\cong\Hom_\C(F^{d-1}H^{2d-2}_B(X),\C)/\Image H_{2d-2}^B(X,\Z(2-d))\label{DB-isom-2}
\end{align}
where $d=\dim X$.

\subsection{Sketch of Proof of Log Formula}
Let $f:X\to \P^1$ be a HG fibration over $\ol\Q$ with multiplication by $(R,e)$. 
Suppose that $\dim X=2$ and there is a section $\P^1\to X$
(e.g. HG fibrations of Gauss type, Example \ref{HG-exmp-gauss}). 
Consider a Cartesian square
\[
\xymatrix{
X_l\ar[rd]_{f_l}\ar[r]^i&X^\prime_l\ar[r]\ar[d]
\ar@{}[rd]|{\square}&X\ar[d]^f\\
&\P^1\ar[r]^{t\to t^l}&\P^1
}
\]
where $i$ is a desingularization.
Let $S:=\P^1\setminus\{t=0,1,\infty\}$, $S_l:=\P^1\setminus\{t^l=0,1,\infty\}$ and
$U_l:=f^{-1}_l(S_l)\subset X_l$ be the complement of singular fibers.
Let $Z:=\cup_i f_l^{-1}(\zeta_l^i)$ be the inverse image of $f^{-1}(1)$.
Note that the local monodromy $T$ at $t=\zeta_l$
on the $e$-part $(R^1f_{l*}\Q)(e)
:=E\ot_{e,R}R^1f_{l*}\Q$ 
is unipotent and $\log(T)$ has the maximal rank by the condition {\bf(c)}.
As is shown in \cite{a-o-2} Proposition 4.8, 
one can construct non-trivial elements
\[
\xi\in H^3_{\cM}(X_l,\Z(2)).
\]
which lie in the image of $H^3_{\cM,Z}(X_l,\Z(2))$.
Suppose $\reg(\xi)\in H^3_{\cD}(X,\Z(2))'$. Then 
\[
\reg(\xi)\in H^2_B(X_l,\C)/F^2H^2(X_l)+H^2_B(X_l,\Z(2))
\]
by the isomorphism \eqref{DB-isom-1}. 
By the natural map $H^2(X_l)\to H^2(U_l)$
we have
\[
\reg(\xi)|_{U_l}\in 
W_2H^2_B(U_l,\C)/\Image F^2H^2(X_l)+H^2_B(X_l,\Z(2))
\]
where $W_\bullet$ denotes the weight filtration.
There is an exact sequence
\[
0\lra H^1(S_l,R^1f_{l*}\Z)\lra H^2(U_l,\Z)\lra H^2(f_l^{-1}(t),\Z)\lra 0
\]
which splits (up to torsion) by a section $\P^1\to X_l$.
Hence we have
\[
\ol\reg(\xi)\in 
(F^1W_2H^1(S_l,R^1f_{l*}\C))^\lor/\Image H_2^B(X_l,\Z)
\]
by \eqref{DB-isom-2}.
Recall that the sheaf $R^1f_{*}\Q$ is endowed with multiplication by
$R$.
For $\zeta_l\in\mu_l$, let $[\zeta_l]$ be the automorphism of $U_l$ given by
$t\to \zeta_lt$. 
Let $\pi:S_l\to S$ be the cyclic covering.
The sheaf $\pi_*R^1f_{l*}\Q=\pi_*\pi^*R^1f_{*}\Q=R^1f_{*}\Q\ot \pi_*\Q$
is endowed with multiplication by the group ring $R[\mu_l]$ in a natural way, and hence 
so is $H^1(S_l,R^1f_{l*}\Q)=H^1(S,\pi_*R^1f_{l*}\Q)$.
Let $\chi:R[\mu_l]\to \ol\Q$ be a homomorphism.
Under a mild assumption, one can show that 
\[
W_2H^1(S,R^1f_{l*}\C)(\chi):=H^1(S,R^1f_{l*}\Q)\ot_{R[\mu_l],\chi}\ol\Q
\]
is one-dimensional (see \cite{a-o-2} \S 4.3 for detail).
Let $\omega_\chi
\in W_2F^1H^1_\dR(S_l,R^1f_{l*}\Omega^\bullet_{U_l/S_l})(\chi)$
be a $\ol\Q$-basis.
The main result of \cite{a-o-2} is the regulator formula 
\begin{equation}\label{sketch-eq1}
\langle\ol\reg(\xi),\omega_\chi\rangle=A_\chi+A'_\chi\cdot B(a_\chi,b_\chi){}_3F_2\left({a_\chi,b_\chi,q_\chi\atop a_\chi+b_\chi,q_\chi+1};1\right)\mod \Image H_2^B(X_l,\Z)
\end{equation}
with some $A_\chi,A'_\chi\in \ol\Q$, $A'_\chi\ne0$, where $a_\chi,b_\chi,q_\chi$ are certain rational numbers 
defined from the monodromy action on $R^1f_{*}\Q$
(see \cite{a-o-2} Theorem 4.7 or \cite{aot-1} Theorem 3.1 for the detail).
On the other hand,
it follows from the theory of Beilinson regulator that
\begin{equation}\label{sketch-eq2}
\langle\ol\reg(\xi),\omega_\chi\rangle\in \log\ol\Q^\times
\end{equation}
if $W_2H^1(S,R^1f_{l*}\Q)(e)$ is a Tate Hodge structure of type $(1,1)$, 
or equivalently the triplet
$(a_\chi,b_\chi,q_\chi)$ satisfy the condition \eqref{main-cond} (\cite{aot-1} Propositions
3.2, 3.3).
In this case, the periods (i.e. the image of $H_2^B(X_l,\Z(2))$) are contained in $2\pi i\ol\Q$.
Thus \eqref{sketch-eq1} and \eqref{sketch-eq2} imply
the log formula \eqref{main-result}.

\section{Explicit Log formula}\label{gen-sect}
To obtain the explicit log formula, we need to compute
\eqref{sketch-eq1} and \eqref{sketch-eq2} explicitly.
One can compute \eqref{sketch-eq2} in terms of elements of the motivic cohomology
(if one knows the generators of the Neron-Severi group $\NS(X_l)$).
On the other hand, to compute the RHS of \eqref{sketch-eq1}, we need to
make ``$A_\chi,A'_\chi$'' clear.  
This is done by constructing a nice rational 2-form 
``$[\omega_\chi]_\Del$''
which shall be given in \eqref{del-prop-1-eq0}. This is the technical heart of this paper.
\subsection{Relative de Rham cohomology}\label{relative-sect}
For a smooth manifold $M$, we denote by 
$\cA^q(M)$ the complex of spaces of smooth differential $q$-forms on $M$
with coefficients in $\C$.

\medskip

Let $X$ be a quasi-projective smooth variety over $\C$.
The de Rham cohomology $H^q_\dR(X)$ is defined to be the cohomology
of the complex $\cA^\bullet(X)$
\[
H^q_\dR(X)=H^q(\cA^\bullet(X)).
\]
By Grothendieck's comparison theorem, one may replace
$\cA^\bullet(X)$ with the algebraic de Rham complex,
\[
H^q(\cA^\bullet(X))\cong H^q_\zar(X,\Omega^\bullet_X).
\]
The right hand side is often referred as algebraic de Rham cohomology groups
(and the left hand side as analytic de Rham cohomology).
In this paper we identify the both sides, and simply call the de Rham cohomology.

In more general, the relative de Rham cohomology groups $H^q_\dR(X_\bullet,Y_\bullet)$ for 
an embedding $Y_\bullet
\hra X_\bullet$ of simplicial schemes are defined (e.g. \cite{hodge III} 8.3.8).
We here review the definition of $H^2_\dR(V,D)$ in case that $V$ is a quasi-projective smooth
surface over $\C$ and $D\subset V$ a reduced curve (i.e. a reduced 
closed subscheme of codimension one). 
Let $\rho:\wt{D}\to D$ be the normalization and $\Sigma\subset D$
the set of singular points. Let $s:\wt{\Sigma}:=\rho^{-1}(\Sigma)
\hra\wt{D}$ be the inclusion.
There is an exact sequence
\[
0\lra \cO_{D}\os{\rho^*}{\lra} \cO_{\wt{D}}\os{s^*}{\lra} 
\C_{\wt{\Sigma}}/\C_\Sigma\lra 0
\]
where $\C_{\wt{\Sigma}}=\mathrm{Maps}(\wt{\Sigma},\C)=\Hom(\Z\wt{\Sigma},\C)$, 
$\rho^*$ and $s^*$ are the pull-back.
We define $\cA^\bullet(D)$ to be the mapping fiber of 
$s^*:\cA^\bullet(\wt{D})\to \C_{\wt{\Sigma}}/\C_\Sigma$:
\[
\cA^0(\wt{D})
\os{s^*\op d}{\lra} \C_{\wt{\Sigma}}/\C_\Sigma\op \cA^1(\wt{D})
\os{0\op d}{\lra} \cA^2(\wt{D})
\]
where the first term is placed in degree 0.
Then
\[
H^q_\dR(D)=H^q(\cA^\bullet(D))
\]
is the de Rham cohomology of $D$, which fits into the exact sequence
\[
\cdots\lra H^0_\dR(\wt{D})\lra \C_{\wt{\Sigma}}/\C_\Sigma
\lra H^1_\dR(D)\lra H^1_\dR(\wt{D})\lra\cdots.
\]
There is a natural pairing 
\begin{equation}\label{pairingD}
H_1(D,\Z)\otimes H^1_\dR(D)\lra \C,
\quad
\gamma\ot z\mapsto\int_\gamma z:=\int_\gamma\eta-c(\partial(\rho^{-1}\gamma))
\end{equation}
where $z=(c,\eta)\in \C_{\wt{\Sigma}}/\C_\Sigma\op\cA^1(\wt{D})$ with $d\eta=0$
and $\partial:H_1(\wt{D},\wt\Sigma)\to H_0(\wt\Sigma)=\Z\wt\Sigma$ denotes the boundary map (note that $c(\partial(\rho^{-1}\gamma)=0$ if $c\in \C_\Sigma$).

\medskip

We define $\cA^\bullet(V,D)$
to be the mapping fiber of 
$j^*:\cA^\bullet(V)\to \cA^\bullet(D)$ the pull-back by $j:D\hra V$:
\[
\cA^0(V)\os{\cD_0}{\lra} \cA^0(\wt{D})\op \cA^1(V)
\os{\cD_1}{\lra} \C_{\wt{\Sigma}}/\C_\Sigma\op\cA^1(\wt{D})\op \cA^2(V)\os{\cD_2}{\lra}\cdots
\]
where 
\[
\cD_0=(j\rho)^*\op d,\quad
\cD_1=\begin{pmatrix}
-(s^*\op d)&0\op(j\rho)^*\\
&d
\end{pmatrix},\quad
\cD_2=\begin{pmatrix}
-(0\op d)&(j\rho)^*\\
&d
\end{pmatrix},\ldots
\]
Then 
\begin{equation}\label{exp-9-0}
H^q_\dR(V,D)=H^q(\cA^\bullet(V,D))
\end{equation}
is the de Rham cohomology which
fits into the exact sequence
\begin{equation}\label{exp-9}
\cdots\lra H^{q-1}_\dR(D)\lra H^q_\dR(V,D)\lra H^q_\dR(V)\lra H^q_\dR(D)\lra\cdots. 
\end{equation}
An arbitrary element of $H^2_\dR(V,D)$ is represented by 
\begin{equation}\label{element}
(c,\eta,\omega)
\in \C_{\wt{\Sigma}}/\C_\Sigma\op\cA^1(\wt{D})\op \cA^2(V)
\end{equation}
which satisfies $j^*\omega=d\eta$ and $d\omega=0$.
They are subject to
relations $(s^*f, df,0)\sim0$ and $(0,j^*\theta,d\theta)\sim0$ for $f\in \cA^0(\wt{D}_0)$
and $\theta\in \cA^1(V)$.
The natural pairing
\begin{equation}\label{pairingVD1}
H_2(V,D;\Z)\ot H_\dR^2(V,D)\lra \C,\quad \Gamma\ot z
\longmapsto \int_\Gamma z
\end{equation}
is given by
\begin{equation}\label{pairingVD2}
\int_\Gamma z:=\int_\Gamma \omega-\int_{\partial\Gamma}(c,\eta)
=\int_\Gamma \omega-\int_{\partial\Gamma}\eta+c(\rho^{-1}(\partial\Gamma)).
\end{equation}
\subsection{The Beilinson regulator map by 1-extensions of mixed Hodge structures}
\label{bei-sect}
Let $X$ be a smooth quasi-projective variety over $\C$.
Let
\[
\reg:H^3_\cM(X,\Z(2))\lra H^3_\cD(X,\Z(2))
\]
be the Beilinson regulator map to the Deligne-Beilinson cohomology group (\cite{schneider}).
We here describe it in terms of 1-extensions of mixed Hodge structures 
(abbreviated to MHS's).
For simplicity we assume that $X$ is a projective smooth surface.
Let $Z\subset X$ be a curve.
There is also the regulator map $\reg_Z$ 
on $H^3_{\cM,Z}(X,\Z(2))$ which fits into a commutative 
diagram
\[
\xymatrix{
H^3_{\cM,Z}(X,\Z(2))\ar[d]\ar[r]^{\reg_Z}& H^3_{\cD,Z}(X,\Z(2))\ar[d]\\
H^3_{\cM}(X,\Z(2))\ar[r]^{\reg}& H^3_{\cD}(X,\Z(2)).
}
\]
Let $\Ext^1(\Z,-)$ denote the group of $1$-extensions of MHS's.
There is a commutative diagram
\[
\xymatrix{
0\ar[r]&\Ext^1(\Z,H_2(Z,\Z))\ar[d]\ar[r]& H^3_{\cD,Z}(X,\Z(2))\ar[r]^{\can}\ar[d]& 
H_1(Z,\Z)\cap H^{0,0}\ar[r]\ar[d]^i&0\\
0\ar[r]&\Ext^1(\Z,H_2(X,\Z))\ar[r]& H^3_\cD(X,\Z(2))\ar[r]& 
H_1(X,\Z)_{\mathrm{tor}}\ar[r]&0
}
\]
with exact rows where $H^{p,q}\subset H(X,\C)$ denotes the Hodge $(p,q)$-component.
We call the composition $c:=\can\circ\reg_Z$ the {\it cycle map}.
The above diagram gives rise to a map
\[
\Phi:\ker(i)\lra \Ext^1(\Z,H_2(X,\Z)/H_2(Z)).
\]
This is explicitly described in the following way.
Let
\[
0\lra H_2(X,\Z)/H_2(Z)\lra H_2(X,Z;\Z)\os{\partial}{\lra} H_1(Z,\Z)
\]
be the exact sequence of homology.
Then, for $\gamma\in H_1(Z,\Z)\cap H^{0,0}$ such that $\gamma\in\ker(i)$
($\Leftrightarrow$ $\gamma\in \Image\partial$), 
$\Phi(\gamma)$ is the 1-extension corresponding to 
\begin{equation}\label{bei-1}
0\lra H_2(X,\Z)/H_2(Z)\lra \partial^{-1}(\Z \gamma)\lra \Z\lra 0.
\end{equation}
Summing up the above we have the following proposition.
\begin{prop}
Write the composition
\[
\ker[H^3_{\cM}(X,\Z(2))\to H_1(X,\Z)_{\mathrm{tor}}]\os{\reg}{\to}
\Ext^1(\Z,H_2(X,\Z))\to  
\Ext^1(\Z,H_2(X,\Z)/H_2(Z))
\]
by $\ol\reg$.
Let $\xi\in H^3_{\cM,Z}(X,\Z(2))$ and  suppose that the homology cycle $\gamma_\xi:=c(\xi)
\in H_1(Z,\Z)$ lies in the image of $\partial$.
Then
$\ol{\reg}(\xi)$ is the 1-extension \eqref{bei-1} for $\gamma=\gamma_\xi$.
\end{prop}
Writing down the 1-extension \eqref{bei-1} in a down-to-earth way,
we also have the following proposition.
\begin{prop}\label{bei-prop-1}
Write $H^2(X)_Z:=\ker[H^2(X)\lra H^2(Z)]$, and consider
the surjective map
$F^1H^2_\dR(X,Z)\to F^1H^2_\dR(X)_Z$.
We fix $(c,\eta,\wt{\omega})\in F^1H^2_\dR(X,Z)$ a lifting
for each $\omega\in F^1H^2_\dR(X)_Z$.
Fix $\Gamma_\xi\in H_2(X,Z;\Z)$ a lifting of $\gamma_\xi$. 
Then under the natural identification
\[
\Ext^1(\Z,H_2(X,\Z)/H_2(Z))\cong 
\Hom(F^1H^2_\dR(X)_Z,\C)/\Image H_2(X,\Z),
\]
the Beilinson regulator is given as follows
\[
\ol{\reg}(\xi)=[\omega\to
\langle\Gamma_\xi,(c,\eta,\wt{\omega})\rangle]
\]
where $\langle\, ,\,\rangle$ denotes the natural pairing $H_2(X,Z;\Z)\ot_\Z H^2_\dR(X,Z)
\to\C$.
\end{prop}
Note
\[
\langle\Gamma_\xi,(\wt{\omega},\eta)\rangle
=\int_{\Gamma_\xi}\wt{\omega}-\int_{\gamma_\xi}(c,\eta)
\]
and this does not depend on the choice of $(c,\eta,\wt{\omega})$ because
$\gamma_\xi\in H^{0,0}$ and hence $\int_{\gamma_\xi}$ annihilates elements of
$F^1H^1_\dR(Z)$.
We should keep notice that,
it is {\it not} true in general that
$\int_{\Gamma_\xi}\wt{\omega}$ depends only on the cohomology class $\omega\in
H^2_\dR(X)$.

\subsection{Deligne's canonical extensions and lifting of differential forms}\label{del-sect}
It is not so simple to compute ``$(\wt{\omega},\eta)$'' 
in Proposition \ref{bei-prop-1} for a given 
$\omega\in F^1H^2_\dR(X)_Z$. 
In the case that $X$ is a fibration of curves and $Z$
is a fibral divisor (i.e. $f(Z)$ are points), there is a nice technique 
developed in \cite{a-real} (see also \cite{a-o-1} Appendix)
to solve the question by using Deligne's canonical extensions.

\medskip

Let $C$ be a smooth projective curve.
We mean by a fibration of curves over $C$ a surjective and projective morphism
$f:X\to C$ with $X$ a nonsingular surface.
Let $S\subset C$ be a Zariski open set such that $f$ is smooth over $S$.
Put $T:=C\setminus S$ and $U:=f^{-1}(S)$.
Then $\cH:=H^1_\dR(U/S)$ is a vector bundle over $S$ endowed with 
the Gauss-Manin connection $\nabla$.
Let $\cH_e$ denote Deligne's canonical extension on $C$, so that the connection extends to
\[
\nabla:\cH_e\lra \Omega^1_C(\log  T)\ot \cH_e
\] 
and the eigenvalues of the residue $\Res(\nabla)$ belong to $[0,1)$.
Let $j:S\hra C$ be the embedding.
One can easily show that the canonical map
\[
[\cH_e\to\Omega^1_C(\log T)\ot \cH_e]\lra
[j_*\cH\to\Omega^1_S\ot j_*\cH]
\]
of complexes of sheaves is a quasi-isomorphism, so that one has the isomorphism
\[
H^1_\dR(C,\cH_e):={\mathbb H}^1_\zar(C,\cH_e\to\Omega^1_C(\log T)\ot \cH_e)\cong 
H^1_\dR(S,\cH)
\hra H^2_\dR(U).
\]
Consider the commutative diagram
\[
\xymatrix{
&0\ar[d]\\
&\Omega^1_C(\log  T)\ot F^1\cH_e\ar[d]\\
F^1\cH_e\ar@{=}[d]\ar[r]^{\nabla\qquad}&\Omega^1_C(\log  T)\ot \cH_e\ar[d]\\
F^1\cH_e\ar[r]^{{\ol\nabla}\hspace{1cm}}&\Omega^1_C(\log  T)\ot \cH_e/F^1\ar[d]\\
&0
}
\]
where $F^1\cH_e:=\cH_e\cap j_*F^1\cH$ with $j:S\hra C$.
Let $C^\circ\subset C$ be a Zariski open set such that 
$\nabla|_{C^\circ}$ is bijective. Put  $X^\circ:=f^{-1}(C^\circ)$.
We assume that $C^\circ\ne\emptyset$.
We do not assume neither $C^\circ\subset S$ nor $C^\circ\supset S$.
Then the above diagram gives rise to an exact sequence
\[
\vg(C^\circ,F^1\cH_e)\os{\nabla}{\lra}\vg(C^\circ,\Omega^1_C(\log  T)\ot \cH_e)
\to
\vg(C^\circ,\Omega^1_C(\log  T)\ot F^1\cH_e)\to0.
\]
We thus have a composition of maps
\begin{align*}
F^1H^1_\dR(S,\cH)&= H^1_\zar(C,F^1\cH_e\to\Omega^1_C(\log  T)\ot \cH_e)\\
&\to H^1_\zar(C^\circ,F^1\cH_e\to\Omega^1_C(\log  T)\ot \cH_e)\\
&\os{\cong}{\to} \vg(C^\circ,\Omega^1_C(\log  T)\ot F^1\cH_e)\\
&\subset \vg(X^\circ\cap U,\Omega^2_X)
\end{align*}
which we denote by $\Theta_\Del$. This is an injective map (\cite{a-real} Prop. 3.10).
Let $W_\bullet=W_\bullet H_\dR(S,\cH)$ denote the weight filtration.
One easily sees that the image of $F^1W_2H^1_\dR(S,\cH)=F^1H^1_\dR(S,\cH)\cap W_2$
lies in the subspace $\vg(X^\circ,\Omega^2_{X^\circ})$ (\cite{a-real} (3.25)),
so that
one also has an injective map
\begin{equation}\label{Theta-del}
\Theta_\Del:F^1W_2H^1_\dR(S,\cH)\lra \vg(X^\circ,\Omega^2_{X^\circ}).
\end{equation}
For $\omega\in F^1W_2H^1_\dR(S,\cH)$, we define
\begin{equation}\label{del-prop-1-eq0}
\omega_\Del:=\Theta_\Del(\omega).
\end{equation}
Let
\[
H^2_\dR(X)_\fib:=\ker[H^2_\dR(X)\lra \prod_{t\in C} H^2_\dR(f^{-1}(t))]
\]
be the subspace perpendicular to all fibral divisors.
We define $H^2_\dR(X^\circ)_\fib$ and $H^2_\dR(U)_\fib$ similarly.
Note $H^2_\dR(U)_\fib\subset H^1_\dR(S,\cH)$.
Then we see
\begin{equation}\label{del-eq-1}
\omega|_{X^\circ}\equiv (\omega|_U)_\Del \quad \mbox{in }H^2_\dR(X^\circ)_\fib
\end{equation}
for $\omega\in F^1H^2_\dR(X)_\fib$.
Indeed $((\omega|_U)_\Del)|_{X^\circ\cap U}
\equiv\omega|_{X\circ\cap U}$ in
$H^2_\dR(X^\circ\cap U)_\fib$ by the definition, and hence \eqref{del-eq-1} follows from the fact that
$H^2_\dR(X^\circ)_\fib\to H^2_\dR(X^\circ\cap U)_\fib$ is injective
(\cite{a-real} Prop. 3.4 (2)).
\begin{prop}[\cite{a-real} Thm. 3.12, \cite{a-o-1} Lem. 7.3]\label{del-prop-1}
Let $Z\subset X^\circ$ be a fibral divisor (i.e. $f(Z)$ are closed points).
Write $H^2_\dR(X^\circ)_Z:=\ker[H^2_\dR(X^\circ)\to H^2_\dR(Z)]$ and consider
\[
\xymatrix{
H^1_\dR(Z)\ar[r]&H^2_\dR(X^\circ,Z)\ar[r]
&H^2_\dR(X^\circ)_Z\ar[r]&0\\
&&H^2_\dR(X^\circ)_\fib\ar[u]^\cup&
}
\]
Assume $C^\circ\ne \emptyset$.
Then for $\omega\in F^1H^2_\dR(X)_\fib$, the element
\begin{equation}\label{del-prop-1-eq1}
(0,0,(\omega|_U)_\Del)\in H^2_\dR(X^\circ,Z)
\end{equation}
is a lifting of $\omega|_{X^\circ}$ and it belongs to
$F^1H^2_\dR(X^\circ,Z)$.
\end{prop}
We can construct $(\omega|_U)_\Del$ only when $C^\circ\ne\emptyset$.
This is satisfied if $f$ has totally degenerate semistable reductions
(\cite{a-real} Lem. 3.7).

\subsection{Explicit Log formula}\label{ELF-sect}
Let $f:X\to \P^1$ be a HG fibration with multiplication by $(R,e)$. 
Suppose $\dim X=2$ for simplicity. Consider the Cartesian square
\[
\xymatrix{
X_l\ar[rd]_{f_l}\ar[r]^i&X^\prime_l\ar[r]\ar[d]
\ar@{}[rd]|{\square}&X\ar[d]^f\\
&\P^1\ar[r]^{t\to t^l}&\P^1
}
\]
where $i$ is a desingularization.
Let $Z:=\cup_i f_l^{-1}(\zeta_l^i)$ be the inverse image of $f^{-1}(1)$,
a totally degenerate semistable fiber.
Let $C=\sum n_i C_i$ be a 1-cycle in $X_l$ with $\Z$-coefficients which is perpendicular to
all components of singular fibers, in other words the cycle class
$\omega_C=\sum n_i\omega_{C_i}\in H^2_\dR(X_l)\cap H^{1,1}$ belongs to $H^2_\dR(X_l)_\fib$.
Let $h_{C_i}:H_2(X_l,\Z)\cong H^2(X_l,\Z(2))\to H^2(C_i,\Z(2))\to \Z(1)$ 
be the composition of the pull-back of the embedding $C_i\to X_l$
and the trace map.
Note that the cycle map $\Z\to H^2(X_l,\Z(1))$, $1\mapsto\omega_{C_i}$ 
coincides with the dual map of $h_{C_i}$ (modulo torsion).
Put $h_C:=\sum n_ih_{C_i}$.
Since $C$ is perpendicular to fibral divisors, $h_C$ factors through $H_2(X_l)/\langle \fib\rangle$
where $\langle \fib\rangle$ denotes the image of $H_2$ of fibral divisors.
Hence
we have a commutative diagram
\begin{equation}\label{ex-comm-1}
\xymatrix{
\Ext^1(\Z,H_2(X_l,\Z)/\langle \fib\rangle)\ar[r]^{\cong\qquad}\ar[d]_{h_C}
& \Hom(H^2_\dR(X_l)_\fib,\C)/\Image H_2(X_l,\Z)\ar[d]^{\omega^\lor_C}\\
\Ext^1(\Z,\Z(1))\ar[r]^\cong& \C/\Z(1)\\
}
\end{equation}
where $\omega^\lor_C$ is the map 
induced from $\C\to H^2_\dR(X_l)_Z$, $1\mapsto \omega_C$.
Let $j:\coprod\wt{C}_i\to \cup C_i\hra X_l$ be the composition of normalization
and the embedding.
Let $T_C=\sum n_i\mathrm{Tr}_{C_i}:\oplus H^2(\wt{C}_i,\Z(2))\to\Z(1)$ be the sum 
of the trace maps.
Let $\tr_{\wt{C}_i}:H^3_\cM(\wt{C}_i,\Z(2))\to
H^1_\cM(\Spec\ol\Q,\Z(1))$ be the transfer map induced from the structure morphism
 $\wt{C}_i\to\Spec\ol\Q$. Put $\tr_C:=\sum n_i \tr_{\wt{C}_i}$.
 Then it follows from the compatibility of the Beilinson regulator maps
 and the fact that the regulator on $H^1_\cM(\Spec\C,\Z(1))\cong\C^\times$ coincides with log that
we have a commutative diagram
\begin{equation}\label{ex-comm-2}
\xymatrix{
&&
\Ext^1(\Z,H_2(X^\circ_l)/\langle \fib\rangle)\ar[d]^{i}\\
H^3_\cM(X_l,\Z(2))\ar[r]^\reg\ar[d]_{j^*}\ar@/^20pt/[rr]^{\ol\reg}
& \Ext^1(\Z,H_2(X_l,\Z))\ar[d]_{j^*}\ar[r]&\Ext^1(\Z,H_2(X_l)/\langle \fib\rangle)\ar[ddl]^{h_C}\\
\bigoplus_i H^3_\cM(\wt{C}_i,\Z(2))\ar[r]\ar[d]_{\tr_C}
& \bigoplus_i \Ext^1(\Z,H^2(\wt{C}_i,\Z(2)))\ar[d]_{T_C}\\
H^1_\cM(\Spec\ol\Q,\Z(1))\ar[r]^\log&\Ext^1(\Z,\Z(1))
}
\end{equation}
where $X_l^\circ$ is as in \S \ref{del-sect}.
Note $Z\subset X_l^\circ$ as $Z$ is a union of 
totally degenerate semistable fibers (\cite{a-real} Lemma 3.7).
Let $\xi\in H^3_{\cM,Z}(X_l,\Z(2))$ such that $\gamma_\xi:=c(\xi)$
lies in the image of $\partial: H_2(X_l^\circ,Z;\Z)\to H_1(Z;\Z)$
where $c: H^3_{\cM,Z}(X_l,\Z(2))\to H_1(Z,\Z)\cap H^{0,0}$ is the cycle map
(cf. \S \ref{bei-sect}).
Let 
\[
e(\gamma_\xi)\in\Ext^1(\Z,H_2(X^\circ_l,\Z)/\langle \fib\rangle)
\]
be the extension data arising from
the exact bottom row of the commutative diagram
\[
\xymatrix{
0\ar[r]& H_2(X_l^\circ)/H_2(Z)\ar[r]& H_2(X^\circ_l,Z)\ar[r]^\partial& H_1(Z)\\
0\ar[r]& H_2(X_l^\circ)/H_2(Z)\ar[r]\ar@{=}[u]
&\partial^{-1}(\Z\gamma_\xi)\ar[r]\ar[u]& \Z\ar[r]\ar[u]_a&0
}
\]
where $a:1\mapsto \gamma_\xi$.
Then we have
\begin{equation}\label{ex-comm-2-1}
\ol{\reg}(\xi)=\pm i(e(\gamma_\xi))\in \Ext^1(\Z,H_2(X_l,\Z)/\langle \fib\rangle).
\end{equation}
On the other hand,
we have
\begin{equation}\label{ex-comm-2-2}
e(\gamma_\xi)=\left[\omega\mapsto\langle\Gamma_\xi,(0,0,(\omega|_U)_\Del)\rangle
=\int_{\Gamma_\xi}(\omega|_U)_\Del\right],\quad \omega\in F^1H^2_\dR(X_l^\circ)_\fib
\end{equation}
by Propositions \ref{bei-prop-1}, \ref{del-prop-1} 
where $\Gamma_\xi\in H_2(X_l^\circ,Z;\Z)$ denotes an arbitrary lifting of $\gamma_\xi$.
Applying the map $h_C$ in \eqref{ex-comm-2} on \eqref{ex-comm-2-1}, 
we have from \eqref{ex-comm-1} and \eqref{ex-comm-2-2} the following theorem:
\begin{thm}\label{ex-thm-1}
Let $\Gamma_\xi\in H_2(X_l^\circ,Z;\Z)$ be a lifting of $\gamma_\xi$. Then
\begin{equation}\label{ex-eq1}
\log \tr_C(j^*\xi)
=\int_{\Gamma_\xi}(\omega_C|_U)_\Del\in\C/\Z(1).
\end{equation}
\end{thm}
As is shown in \cite{a-o-1} Proposition 2.6 (ii) or \cite{a-o-2} \S 7.4, the last term 
of \eqref{ex-eq1} is written in terms of the special values
of ${}_3F_2$ at $x=1$. 

\section{Examples of Explicit Log Formula}\label{rec-sect}
In this section, we demonstrate how to prove 
\begin{equation}\label{rec-eq0}
{}_3F_2\left(
{\frac{1}{6},\frac{5}{6},\frac{1}{2}\atop 1,\frac{3}{2}};1\right)=\frac{3\sqrt{3}}{2\pi}\log(2+\sqrt{3}).
\end{equation}

\medskip

Let $f:X\to \P^1$ be an elliptic fibration whose generic fiber
$f^{-1}(t_0)$ is defined by the affine equation
\[
y^2=2x^3-3x^2+t_0.
\]
This is a HG fibration with multiplication by $(\Q,\mathrm{id})$
in the sense of \S \ref{HG-sect} (cf. Example \ref{HG-exmp}).
Let $l\geq 1$ be an integer.
Let $f_l:X_l\to\P^1$ be an elliptic fibration defined by the affine equation
$y^2=2x^3-3x^2+t^l$ with $t^l=t_0$. 
 
The elliptic fibration $f_l$ is endowed with an action of $\mu_l$ the group of $l$-th roots of
$1$. Namely, to $\zeta\in \mu_l$ we associate
 $\sigma_\zeta\in \Aut(X_l)$
an automorphism defined by $\sigma(x,y,t)=(x,y,\zeta t)$.
We thus have $\mu_l\hra\Aut(X_l)$ and $\Q[\mu_l]\hra \End(R^1f_{l*}\Q)$.
Let
\[
M_l:=H^2(X_l,\Q)/\langle\mbox{fibral divisors},\infty\rangle\cong 
W_2H^1(\P^1\setminus\{0,1,\ldots,\zeta_l^{l-1},\infty\},R^1f_{l*}\Q)
\]
where $\infty\subset X_l$ denotes the section $y=\infty$.
For a projector $e:\Q[\mu_l]\to F$ onto a number field $F$, we denote by 
$M_l(e):=F\ot_{e,\Q[\mu_l]}M_l$ the $e$-part.
One easily shows, 
\begin{equation}\label{dim-formula}
\dim_FM_l(e)=
\begin{cases}
1& l/d\ne1,6\\
0&l/d=1,6
\end{cases}\quad d:=\sharp\ker[e:\mu_l\to F^\times].
\end{equation}
This implies $\dim_F(M_l(e)\cap H^{0,0})\leq 1$, and then
\begin{equation}\label{dim-formula-1}
M_l(e)\cap H^{0,0}\ne0
\,\Leftrightarrow\,
F^2M_l(e)=F^2H^2_\dR(X_l)(e)=0
\,\Leftrightarrow\,
2\leq l/d\leq 5.
\end{equation} 
Let $Z$ be the union of totally degenerate semistable fibers over $t^l=1$, and consider elements
\[
\xi_j:=\left(\frac{y-\sqrt{3}(x-1)}{y+\sqrt{3}(x-1)},f_l^{-1}(\zeta_l^j)\right)
\in H_{\cM,Z}^3(X_l,\Z(2)),\quad j\in\{0,1,\ldots,l-1\}.
\]
It is straightforward to see that
$c(\xi_j)\in H_1(f_l^{-1}(\zeta_l^j),\Z)\cong \Z$ is a basis
where $c:H^3_{\cM,Z}(X_l,\Z(2))\to H^3_Z(X_l,\Z(2))=H_1(Z,\Z)$ is the cycle map.

\medskip

To prove \eqref{rec-eq0}
we apply Theorem \ref{ex-thm-1} \eqref{ex-eq1} to the elliptic fibration $f_l$
in case that $l=2$ and
$e:\Q[\mu_2]\to\Q$ is the projector such that $e(\sigma_{-1})=-1$
($\Leftrightarrow$ $d=1$).
Put $\xi:=\xi_0$. By \eqref{dim-formula} and \eqref{dim-formula-1}, 
\begin{equation}\label{dim-formula-2}
M_2(e)=M_2=W_2H^1(\P^1\setminus\{0,\pm1,\infty\},R^1f_{2*}\Q)\cong \Q,
\end{equation}
and this is a Tate-Hodge structure of type $(1,1)$ (and hence generated by a cycle class).

\medskip

\noindent{\bf Step 1}.
The 1st step is to find a (nontrivial) divisor $C$ which is perpendicular to all fibral divisors
and generates the $e$-part $M_2(e)$.
Let 
\[
C_1:x=0,\,y=t,\quad
C_2:x=0,\,y=-t
\]
be sections in $X_2$.
Then $\sigma_{-1}(C_1)=C_2$, and hence the cycle class $[C_1]-[C_2]\in H^2(X_2)$ belongs to the $e$-part.
Let $f_2^{-1}(\infty)=F_1+F_2+F_3+2(F_4+F_5+F_6)+3F_7$ be the singular fiber at $t=\infty$
(see the figure in below).
Put
\[
C:=3(C_1-C_2)+2(F_1-F_2)+F_4-F_5.
\]
Then this is perpendicular to all fibral divisors (see the following figure),
and $M_2(e)=\Q[C]$.

\vspace{1cm}

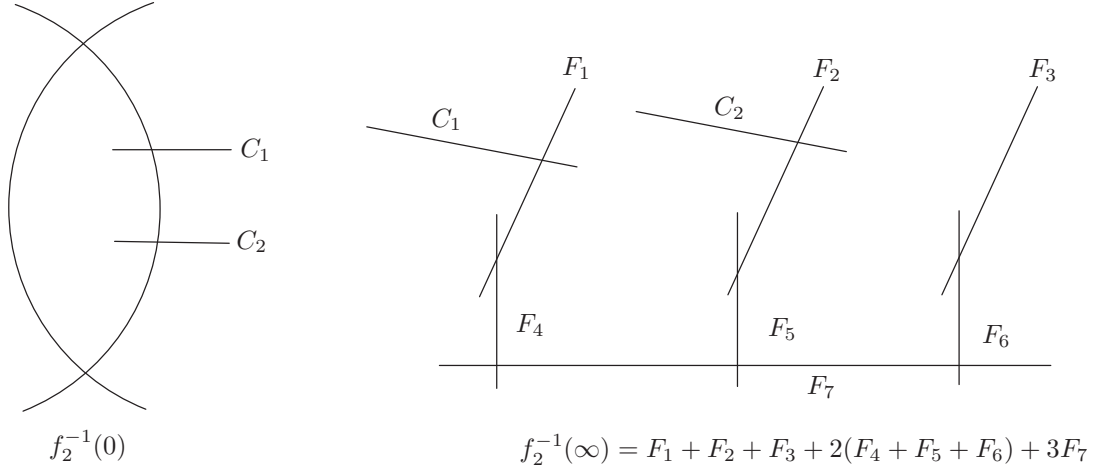
\begin{center}
{\unitlength 0.1in%
\begin{picture}( 54.2000, 22.5000)(  6.6000,-28.3000)%
%
\special{pn 8}%
\special{ar 1764 1642 1138 1138  1.9494474  4.3716705}%
%
\special{pn 8}%
\special{ar 280 1652 1138 1138  5.0531074  1.1921452}%
%
\special{pn 8}%
\special{pa 1170 1340}%
\special{pa 1790 1340}%
\special{fp}%
%
\special{pn 8}%
\special{pa 1180 1820}%
\special{pa 1780 1830}%
\special{fp}%
%
\special{pn 8}%
\special{pa 2880 2470}%
\special{pa 6080 2470}%
\special{fp}%
%
\special{pn 8}%
\special{pa 3180 1680}%
\special{pa 3180 2590}%
\special{fp}%
%
\special{pn 8}%
\special{pa 4440 1670}%
\special{pa 4440 2580}%
\special{fp}%
%
\special{pn 8}%
\special{pa 5600 1660}%
\special{pa 5600 2570}%
\special{fp}%
%
\special{pn 8}%
\special{pa 3090 2110}%
\special{pa 3590 1020}%
\special{fp}%
%
\special{pn 8}%
\special{pa 4390 2100}%
\special{pa 4890 1010}%
\special{fp}%
%
\special{pn 8}%
\special{pa 5510 2100}%
\special{pa 6010 1010}%
\special{fp}%
%
\special{pn 8}%
\special{pa 2500 1220}%
\special{pa 3600 1430}%
\special{fp}%
%
\special{pn 8}%
\special{pa 3910 1140}%
\special{pa 5010 1350}%
\special{fp}%
\put(19.2000,-13.4000){\makebox(0,0){$C_1$}}%
\put(19.0000,-18.2000){\makebox(0,0){$C_2$}}%
\put(29.2000,-11.8000){\makebox(0,0){$C_1$}}%
\put(44.0000,-11.2000){\makebox(0,0){$C_2$}}%
\put(36.0000,-9.1000){\makebox(0,0){$F_1$}}%
\put(49.1000,-9.2000){\makebox(0,0){$F_2$}}%
\put(10.4000,-28.9000){\makebox(0,0){$f^{-1}_2(0)$}}%
\put(60.4000,-9.2000){\makebox(0,0){$F_3$}}%
\put(32.8000,-22.0000){\makebox(0,0)[lt]{$F_4$}}%
\put(46.0000,-22.2000){\makebox(0,0)[lt]{$F_5$}}%
\put(57.2000,-22.6000){\makebox(0,0)[lt]{$F_6$}}%
\put(48.8000,-25.8000){\makebox(0,0){$F_7$}}%
\put(33.0000,-28.3000){\makebox(0,0)[lt]{$f_2^{-1}(\infty)=F_1+F_2+F_3+2(F_4+F_5+F_6)+3F_7$}}%
\end{picture}}%

\end{center}

\vspace{1cm}

\noindent{\bf Step 2} (Computing LHS of \eqref{ex-eq1}).
\begin{align*}
\mbox{LHS of \eqref{ex-eq1}}
&=3\log
\left(\frac{y-\sqrt{3}(x-1)}{y+\sqrt{3}(x-1)}|_{f_2^{-1}(1)\cap C_1}\right)
\left(\frac{y-\sqrt{3}(x-1)}{y+\sqrt{3}(x-1)}|_{f_2^{-1}(1)\cap C_2}\right)^{-1}\\
&=3\log
\left(\frac{1+\sqrt{3}}{1-\sqrt{3}}\right)
\left(\frac{-1+\sqrt{3}}{-1-\sqrt{3}}\right)^{-1}
\\
&=6\log(2+\sqrt{3}).
\end{align*}

\medskip

\noindent{\bf Step 3} (Computing $(\omega_C|_U)_\Del$).
Let $S:=\P^1\setminus\{0,\pm 1,\infty\}$ and put $U:=f_2^{-1}(S)$.
Let $X_2^\circ=f^{-1}_2(\P^1\setminus\{\infty\})$ be as in \S \ref{del-sect}.
Let $\omega_C\in H^2_\dR(X_2)_\fib$ be the cycle class. Then we claim
\begin{equation}\label{rec-eq1}
(\omega_C|_U)_\Del=\alpha
dt\frac{dx}{y}\in \vg(X_2^\circ,\Omega^2_{X_2}),\quad \exists \alpha\in \C^\times.
\end{equation}
This is proven in the following way.
Let $\cH:=H^1_\dR(U/S)$ be the vector bundle on 
$S$ equipped with the Gauss-Manin connection $\nabla$.
By \eqref{dim-formula-2},
$W_2H^1_\dR(S,\cH)=F^1W_2H^1_\dR(S,\cH)$ is one-dimensional
and moreover it is spanned by the cycle class $\omega_C|_U$ under the inclusion
$H^2_\dR(X_2)_\fib\hra W_2H^1_\dR(S,\cH)$. 
Note that $(\omega_C|_U)_\Del\ne0$ as 
$\Theta_\Del$ is injective (see \eqref{Theta-del}).
Hence
\begin{equation}\label{rec-eq10}
\Image[\Theta_\Del:F^1W_2H^1_\dR(S,\cH)\to \vg(X^\circ_2,\Omega^2_{X^\circ_2})]
=\C (\omega_C|_U)_\Del.
\end{equation}
On the other hand, we claim
\begin{equation}\label{rec-eq11}
\Image[\Theta_\Del:F^1W_2H^1_\dR(S,\cH)\to \vg(X^\circ_2,\Omega^2_{X^\circ_2})]
=\C dt\frac{dx}{y}.
\end{equation}
The explicit description of $\nabla$ is given as follows (e.g. \cite{a-real} Theorem 6.4)
\begin{equation}\label{rec-eq2}
\begin{pmatrix}
\nabla\left(\frac{dx}{y}\right)&\nabla\left(\frac{xdx}{y}\right)
\end{pmatrix}
=
\begin{pmatrix}
\frac{dx}{y}&\frac{xdx}{y}
\end{pmatrix}A,\quad
A:=\frac{dt_0}{6(t_0-t_0^2)}\begin{pmatrix}
t_0&t_0\\
-1&-t_0
\end{pmatrix}
\end{equation}
where $t_0=t^2$.
Deligne's extension $\cH_e$ of $\cH$ is given by 
a local frame $\{dx/y,xdx/y\}$ on $\P^1\setminus\{\infty\}$ and
$\{dx/y,t^{-1}xdx/y\}$ on a neighborhood of $t=\infty$.
Indeed one easily check that
\[
\nabla(\cH_e)\subset \Omega^1_{\P^1}(\log T)\ot\cH_e,\quad T:=\{0,\pm 1,\infty\}
\]
and any eigenvalue of $\Res(\nabla)$ at a point of $T$ is $0,1/6$ or $5/6$.
Since $F^1\cH_e\cong \O_{\P^1}$ and $\cH_e/F^1\cH_e\cong \O_{\P^1}(-1)$,
one has an exact sequence
\[
0\to H^0(F^1\cH_e)\to H^0(\Omega^1_{\P^1}(\log T)\ot\cH_e)\to
F^2H^1_\dR(S,\cH)\to 0
\]
and $F^2W_2H^1_\dR(S,\cH)$ is generated by
\[
\eta:=\frac{dt}{t(t^2-1)}\left(\frac{t^2dx}{y}-\frac{xdx}{y}\right).
\]
Noticing
\[
\nabla\left(t\frac{dx}{y}\right)=dt\frac{dx}{y}-
\frac{dt}{6t(t^2-1)}\left(\frac{t^2dx}{y}-\frac{xdx}{y}\right)
\]
by \eqref{rec-eq2}, we have 
\[\Theta_\Del(\eta)=6dt\frac{dx}{y}\]
by definition of $\Theta_\Del$.
This shows \eqref{rec-eq11}. Now \eqref{rec-eq1} is immediate from \eqref{rec-eq10} and
\eqref{rec-eq11}.

\medskip

The coefficient ``$\alpha$'' shall be determined in Step 5.
Before this, we show a certain property of $\alpha$.

Let $\delta_t\in H_1(f_2^{-1}(t),\Z)$ be the vanishing cycle at $t=1$, namely
$\delta_t$ is a homology 1-cycle which is a generator of
$\ker[H_1(f_2^{-1}(t),\Z)\to H_1(f_2^{-1}(1),\Z)]\cong \Z$.
Then it defines a Lefschetz thimble $\Delta$ over $[0,1]\subset \P^1(\C)$, and
hence a homology 2-cycle $(1-\sigma_{-1})\Delta\in H_2(X_2^\circ,\Z)$.
Since $C|_{X_2^\circ}$ is a divisor with integral coefficients, one has
$\omega_C|_{X_2^\circ}\in H^2(X^\circ_2,\Z(1))$ and hence
\begin{equation}\label{rec-eq3}
\int_{(1-\sigma_{-1})\Delta}(\omega_C|_U)_\Del
=\int_{(1-\sigma_{-1})\Delta}\omega_C|_{X_2^\circ}\in\Z(1)
\end{equation}
by \eqref{del-eq-1}. 
\begin{lem}\label{rec-lem1}
\[
\int_{\delta_t}\frac{dx}{y}=\frac{2\pi i}{\sqrt{3}}\,{}_2F_1\left(\frac{1}{6},\frac{5}{6},1;1-t^2\right)
\]
\end{lem}
\begin{proof}
Let $D_{t_0}=\nabla_{\frac{d}{dt_0}}$ be the composition $\cH\to \Omega^1_S\ot\cH\to \cH$ where
the second arrow given by $dt_0\ot v\mapsto v$.
One can derive from \eqref{rec-eq2} that
\[
\left((t_0-t_0^2)D_{t_0}^2+(1-2t_0)D_{t_0}-\frac{5}{36}\right)
\left(\frac{dx}{y}\right)=0.
\]
This implies that $\int_{\delta_t}\frac{dx}{y}$ is a solution of the differential equation
\[
(t_0-t_0^2)\frac{d^2u}{dt_0^2}+(1-2t_0)\frac{du}{dt_0}-\frac{5}{36}u=0.
\]
Therefore $\int_{\delta_t}\frac{dx}{y}$ is a $\C$-linear combination of
\[
{}_2F_1\left(\frac{1}{6},\frac{5}{6},1;1-t_0\right),
\quad{}_2F_1\left(\frac{1}{6},\frac{5}{6},1;t_0\right).
\]
Since $\delta_t$ is invariant by the local monodromy at $t_0=1$,
there is a constant $K\in \C$ such that
\[
\int_{\delta_t}\frac{dx}{y}=K\cdot
{}_2F_1\left(\frac{1}{6},\frac{5}{6},1;1-t_0\right).
\]
One can compute the constant $K$ in the following way.
Let $2x^3-3x^2+t^2=2(x-\alpha_t)(x-\beta_t)(x-\gamma_t)$ where 
$\alpha_t\to-\frac{1}{2}$ and $\beta_t,\gamma_t\to 1$ as $t\to 1$.
Then 
\begin{align*}
K&=\lim_{t\to 1}\int_{\delta_t}\frac{dx}{y}\\
&=\lim_{t\to 1}2\int_{\beta_t}^{\gamma_t}
\frac{dx}{\sqrt{2(x-\alpha_t)(x-\beta_t)(x-\gamma_t)}}\\
&=\lim_{t\to 1}\sqrt{2}i\int_0^{\gamma_t-\beta_t}
\frac{dx}{\sqrt{(x+\beta_t-\alpha_t)x(\gamma_t-\beta_t-x)}}\\
&=\lim_{t\to 1}\sqrt{2}i\int_0^1
\frac{dx}{\sqrt{((\gamma_t-\beta_t)x+\beta_t-\alpha_t)x(1-x)}}\\
&=\sqrt{2}i\int_0^1
\frac{dx}{\sqrt{\frac{3}{2}x(1-x)}}\\
&=\frac{2\pi i}{\sqrt{3}}.
\end{align*}
\end{proof}
Now one computes
\begin{align*}
\mbox{RHS of \eqref{rec-eq3}}&=2\alpha\int_0^1dt\int_{\delta_t}\frac{dx}{y}&\\
&=\frac{4\pi i\alpha}{\sqrt{3}}\int_0^1{}_2F_1\left(\frac{1}{6},\frac{5}{6},1;1-t^2\right)dt
&(\mbox{by Lemma \ref{rec-lem1}})\\
&=\frac{2\pi i\alpha}{\sqrt{3}}\int_0^1t^{-\frac{1}{2}}
{}_2F_1\left(\frac{1}{6},\frac{5}{6},1;1-t\right)dt\\
&=\frac{4\pi i\alpha}{\sqrt{3}}\cdot 
{}_3F_2\left({1,\frac{1}{6},\frac{5}{6}\atop \frac{3}{2},1};1\right)
&(\mbox{by \cite{NIST} 16.5.2})\\
&=\frac{4\pi i\alpha}{\sqrt{3}}\cdot 
{}_2F_1\left({\frac{1}{6},\frac{5}{6}\atop \frac{3}{2}};1\right)\\
&=\frac{4\pi i\alpha}{\sqrt{3}}\frac{\Gamma(\frac{3}{2})\Gamma(\frac{1}{2})}
{\Gamma(\frac{3}{2}-\frac{1}{6})\Gamma(\frac{3}{2}-\frac{5}{6})}
& (\mbox{by \cite{NIST} 15.4.20})\\
&=3\pi i \alpha &(\mbox{by \cite{NIST} 5.5.6}).
\end{align*}
Hence
\begin{equation}\label{rec-eq4}
\alpha\in\frac{2}{3}\Z.
\end{equation}

\medskip

\noindent{\bf Step 4} (Computing RHS of \eqref{ex-eq1}).
Let $\gamma_\xi=c(\xi)\in H_1(f_2^{-1}(1),\Z)$ where 
$c:H^3_{\cM,f_2^{-1}(1)}(X_2,\Z(2))\to H^3_{f_2^{-1}(1)}(X_2,\Z(2))\cong H_1
(f_2^{-1}(1),\Z)$ is the cycle map.
For $0\leq t\leq 1$,
let $\gamma_t\in H_1(f_2^{-1}(t),\Z)$ be the homology cycle such that $\gamma_t|_{t=1}=\gamma_\xi$ and
$\gamma_t|_{t=0}=0$ the vanishing cycle at $t=0$.
The family of $\{\gamma_t\}_t$ defines a Lefschetz thimble $\Gamma_\xi$ over the line segment $[0,1]\subset \P^1(\C)$.
It defines a homology cycle
$\Gamma_\xi\in H_2(X_2^\circ,Z;\Z)$ with boundary 
$\partial\Gamma_\xi=\gamma_\xi=c(\xi)$.
Note that the homology cycle $\gamma_\xi\in H_1(f^{-1}_2(1),\Z)\cong \Z$ is a generator.
The figure of the cycle $\Gamma_\xi$ is as follows, where the orientation of $\gamma_t$
is given by 
either the red arrow or the blue one (we omit to determine the orientation
since it is not necessary in the discussion below).

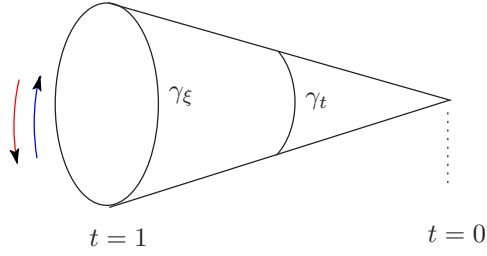
\begin{center}
{\unitlength 0.1in%
\begin{picture}( 23.0100, 11.3000)(  3.9900,-12.6000)%
%
\special{pn 8}%
\special{ar 900 660 270 530  0.6053366  0.6027865}%
%
{\color{blue}\special{pn 8}%
\special{ar 780 760 260 550  2.8068558  3.5985514}%
%
\special{pn 8}%
\special{pa 544 544}%
\special{pa 550 518}%
\special{fp}%
\special{sh 1}%
\special{pa 550 518}%
\special{pa 516 578}%
\special{pa 538 570}%
\special{pa 554 587}%
\special{pa 550 518}%
\special{fp}%
\special{pa 550 518}%
\special{pa 550 518}%
\special{fp}}%
%
{\color{red}\special{pn 8}%
\special{ar 674 776 260 550  2.8068558  3.5985514}%
%
\special{pn 8}%
\special{pa 426 946}%
\special{pa 430 966}%
\special{fp}%
\special{sh 1}%
\special{pa 430 966}%
\special{pa 437 897}%
\special{pa 420 914}%
\special{pa 397 905}%
\special{pa 430 966}%
\special{fp}}%
%
\special{pn 8}%
\special{pa 910 130}%
\special{pa 2700 640}%
\special{fp}%
\special{pa 915 1200}%
\special{pa 2690 640}%
\special{fp}%
\put(8.1000,-14.0500){\makebox(0,0)[lb]{$t=1$}}%
\put(25.8000,-13.8000){\makebox(0,0)[lb]{$t=0$}}%
%
\special{pn 8}%
\special{pa 2685 705}%
\special{pa 2685 1105}%
\special{dt 0.045}%
\put(12.2500,-6.7000){\makebox(0,0)[lb]{$\gamma_\xi$}}%
%
\special{pn 8}%
\special{ar 1645 650 240 345  5.4044040  0.9307082}%
\put(19.4000,-6.8500){\makebox(0,0)[lb]{$\gamma_t$}}%
\end{picture}}%

\end{center}
\begin{lem}\label{rec-lem2}
\[
\int_{\gamma_t}\frac{dx}{y}=\pm\frac{2\pi }{\sqrt{3}}\,{}_2F_1\left(\frac{1}{6},\frac{5}{6},1;t^2\right)
\]
\end{lem}
\begin{proof}
Similar to the proof of Lemma \ref{rec-lem1} (details are left to the reader).
\end{proof}
We now have
\begin{align*}
\mbox{RHS of \eqref{ex-eq1}}&=
\alpha\int_{\Gamma_\xi}dt\frac{dx}{y}& \mbox{(by \eqref{rec-eq1})}\\
&=\alpha\int_0^1dt\int_{\gamma_t}\frac{dx}{y}\\
&=\pm\frac{2\pi \alpha}{\sqrt{3}}\int_0^1{}_2F_1\left(\frac{1}{6},\frac{5}{6},1;t^2\right)dt
&\mbox{(by Lemma \ref{rec-lem2})}\\
&=\pm\frac{\pi \alpha}{\sqrt{3}}\int_0^1t^{-\frac{1}{2}}{}_2F_1\left(\frac{1}{6},\frac{5}{6},1;t\right)dt\\
&=\pm\frac{2\pi \alpha}{\sqrt{3}}\,{}_3F_2\left({\frac{1}{6},\frac{5}{6},\frac{1}{2}\atop
1,\frac{3}{2}};1\right)&\mbox{(by \cite{NIST} 16.5.2)}.
\end{align*}

\medskip

\noindent{\bf Step 5}iFinal Step).
We apply Theorem \ref{ex-thm-1} to the results in Step 2 and Step 4, and hence we have
\[
\alpha\cdot
{}_3F_2\left(
{\frac{1}{6},\frac{5}{6},\frac{1}{2}\atop 1,\frac{3}{2}};1\right)
=\pm\frac{3\sqrt{3}}{\pi}\log(2+\sqrt{3})\in \C/\Z(1).
\]
Taking the absolute value of the real part we have
\[
|\mathrm{Re}(\alpha)|\cdot
{}_3F_2\left(
{\frac{1}{6},\frac{5}{6},\frac{1}{2}\atop 1,\frac{3}{2}};1\right)
=\frac{3\sqrt{3}}{\pi}\log(2+\sqrt{3})\in \R,
\]
\[
(\Longrightarrow\quad\mathrm{Re}(\alpha)=\pm 2.0000000 \quad\mbox{by the aid of computer}.)
\]
Since $\alpha\in \frac{2}{3}\Z$ by \eqref{rec-eq4} this yields $|\mathrm{Re}(\alpha)|=|\alpha|=2$.
This completes the proof of \eqref{rec-eq0}.

\bigskip

\begin{center}
{\bf Other Examples}
\end{center}
If $a=\frac{1}{6}$ and $b=\frac{5}{6}$, then \eqref{main-cond} is satisfied if and only if
$q=\frac{1}{2},\frac{i}{3},\frac{j}{4}$ or $\frac{k}{5}$ where $i\in\{1,2\}$,
$j\in\{1,2,3\}$ and $k\in\{1,2,3,4\}$.
In these cases, the explicit log formulas can be obtained 
by applying the same discussion as above to the elliptic fibration $y^2=2x^3-3x^2+t^l$
where $l=2,3,4,5$ respectively.

\medskip

In case $l=3$, the second author obtained in \cite{yabu} 
\[
{}_3F_2\left({\frac{1}{6},\frac{5}{6},\frac{1}{3}\atop1,\frac{4}{3}};1\right)
=\frac{\sqrt{3}\sqrt[3]{2}}{2\pi}A
-\frac{\sqrt[3]{2}}{\pi}B,
\]
\[
{}_3F_2\left({\frac{1}{6},\frac{5}{6},\frac{2}{3}\atop1,\frac{5}{3}};1\right)
=\frac{\sqrt{3}\sqrt[3]{4}}{3\pi}A
+\frac{2\sqrt[3]{4}}{3\pi}B
\]
where 
\[
A:=
\log\left((1-2^{-\frac{2}{3}})^2+(1+2^{-\frac{2}{3}}\sqrt{3})^2\right)
-\log\left((1-2^{-\frac{2}{3}})^2+(1-2^{-\frac{2}{3}}\sqrt{3})^2\right),
\]
\[
B:=\mathrm{Tan}^{-1}\left(\frac{3}{3+\sqrt[3]{2}+3\sqrt[3]{4}}\right).
\]
In case $l=4$ we have
\[
\frac{2\pi}{12^{3/4}}~ {}_3F_2\left({\frac{1}{6},\frac{5}{6},\frac{1}{4}\atop
1,\frac{5}{4}};1\right)=
\frac{1}{2}\log\left(\frac{3^{5/4}-3^{3/4}+\sqrt{2}}{3^{5/4}-3^{3/4}-\sqrt{2}}\right)
-\mathrm{Cos}^{-1}\left(\frac{3^{5/4}+3^{3/4}}{2\sqrt{5+3\sqrt{3}}}\right),
\]
\[
\frac{7\sqrt{3}}{9}~
\frac{2\pi}{12^{3/4}}~ {}_3F_2\left({\frac{1}{6},\frac{5}{6},\frac{3}{4}\atop
1,\frac{7}{4}};1\right)=
\frac{1}{2}\log\left(\frac{3^{5/4}-3^{3/4}+\sqrt{2}}{3^{5/4}-3^{3/4}-\sqrt{2}}\right)
+\mathrm{Cos}^{-1}\left(\frac{3^{5/4}+3^{3/4}}{2\sqrt{5+3\sqrt{3}}}\right).
\]
In case $l=5$, let $\zeta=e^{2\pi i/5}$, 
 $\zeta_{20}=e^{2\pi i/20}$, $\alpha=1/\sqrt[10]{24}>0$ and 
 \[
e_j:=\frac
{\sqrt{2}\alpha^3 \zeta_{20}^{3}\zeta^j+\frac{\sqrt{2}}{4}\alpha^{-3} \zeta_{20}^{-3}\zeta^j-\sqrt{3}(\alpha^2 \zeta_{20}^{2}\zeta^j-1)}
{\sqrt{2}\alpha^3 \zeta_{20}^{3}\zeta^j+\frac{\sqrt{2}}{4}\alpha^{-3} \zeta_{20}^{-3}\zeta^j+\sqrt{3}(\alpha^2 \zeta_{20}^{2}\zeta^j-1)}
\in \C,\quad j\in \Z.
\]
Put
\[
A_k:=\frac{\Gamma(k/5+1/6)\Gamma(k/5+5/6)}{\Gamma(k/5)^2},
\]
\[
f_k:=\frac{2\pi A_k}{k}\cdot{}_3F_2\left({\frac{1}{6},\frac{5}{6},\frac{k}{5}\atop
1,1+\frac{k}{5}};1\right),\quad k=1,2,3,4.
\]
Note $A_k\in \ol\Q$.
Then 
\[
\frac{5}{\zeta^{2k}-1}f_k=(\zeta^{2k}-1)\log e_0+(\zeta^{2k}-\zeta^{3k})\log e_1
+(\zeta^{2k}-\zeta^k)\log e_2+(\zeta^{2k}-\zeta^{4k})\log e_3+4\pi i\zeta^{2k}
\]
for $k=1,2,3,4$ where $\log(x)$ takes the principal values, 
\[
\log(x)=\log|x|+\mathrm{arg}(x)i\ \left(-\pi < \mathrm{arg}(x) \leq \pi \right).
\]

\end{document}